\documentclass[11pt,english]{amsart}
\usepackage[a4paper]{geometry}
\usepackage{amssymb}
\usepackage{amsmath}
\usepackage{fancyhdr}
\usepackage{color}
\usepackage{enumerate}
\newtheorem{theorem}{Theorem}[section] 
\newtheorem{lemma}[theorem]{Lemma}     

\newtheorem{proposition}[theorem]{Proposition}

\newtheorem{definition}[theorem]{Definition}
\newtheorem{remark}[theorem]{Remark}

\newtheorem*{theorem*}{Theorem}

\usepackage[bookmarksopen, naturalnames]{hyperref}

\makeatother

\begin{document}

\title[ ]{On doubly universal functions}
\author{A. Mouze}
\address{Augustin Mouze, Laboratoire Paul Painlev\'e, UMR 8524, Current address: \'Ecole Centrale de
Lille, Cit\'e Scientifique, CS20048, 59651 Villeneuve d'Ascq cedex}
\email{Augustin.Mouze@math.univ-lille1.fr}


\begin{abstract} Let $(\lambda_n)$ be a strictly increasing sequence of positive integers. Inspired by the notions of topological multiple recurrence and disjointness in dynamical systems, Costakis and Tsirivas have recently established that there exist power series 
$\sum_{k\geq 0}a_kz^k$ with radius of convergence 1 such that the pairs of partial sums 
$\{(\sum_{k=0}^na_kz^k,\sum_{k=0}^{\lambda_n}a_kz^k): n=1,2,\dots\}$ approximate all pairs of polynomials uniformly on 
compact subsets $K\subset\{z\in\mathbb{C} :\vert z\vert>1\},$ with connected complement, if and only if 
$\limsup_{n}\frac{\lambda_n}{n}=+\infty.$ In the present paper, we give a new proof of this statement avoiding the use of advanced tools of potential theory. It allows to obtain the algebraic genericity of the set of such power series and to study the case of doubly universal infinitely differentiable functions. 
Further we show that the Ces\`aro means of partial sums of power series with radius of convergence 1 cannot be frequently universal.   
\end{abstract}

\maketitle

\section{Introduction} For a simply connected domain $\Omega\subset\mathbb{C},$ we will denote by $H(\Omega)$ 
the space of all holomorphic functions on $\Omega.$ Let $\mathbb{D}:=\{z\in\mathbb{C}:\vert z\vert<1\}.$ 
For $f\in H(\mathbb{D}),$ we denote by $S_n(f)$ the $n$-th partial sum of its Taylor development with center $0.$ 
In 1996 Nestoridis proved that there exist functions $f\in H(\mathbb{D})$ such that for every compact set $K\subset \mathbb{C}$ 
with $K^c$ connected and $K\cap \mathbb{D}=\emptyset$ and for every function $h\in A(K),$ where 
$A(K):=H(\mathring{K})\cap C(K),$ there exists a sequence of positive integers $(\lambda_n)$ such that 
$\sup_{z\in K}\vert S_{\lambda_n}(f)(z)-h(z)\vert\rightarrow 0$ as $n\rightarrow +\infty$ \cite{Nes}. Such functions are called {\it universal Taylor series}. The partial sums of its Taylor development diverge in a maximal way. In the following, the set of 
universal Taylor series will be denoted by $\mathcal{U}(\mathbb{D},0).$ We refer the reader to \cite{bgnp} and the references therein for its properties. In particular we know that $\mathcal{U}(\mathbb{D},0)$ is a $G_\delta$ dense subset of $H(\mathbb{D}),$ endowed 
with the topology of uniform convergence on all compact subsets of $\mathbb{D},$ and contains a dense vector subspace apart from $0$. Notice that we know $C^\infty$ versions 
of Nestoridis result (see for instance \cite{bgnp,CostaMaria,ge,MouNes,Pal}). Inspired by the notion of topological multiple recurrence and disjointness in dynamical systems, Costakis and Tsirivas introduced the following new form of universality \cite{CT}.
\begin{definition}\label{def_CT}{\rm Let $(\lambda_n)$ be a strictly increasing sequence of positive integers. 
A function $f\in H(\mathbb{D})$ belongs to the class $\mathcal{U}(\mathbb{D},(\lambda_n),0)$ if 
for every compact set $K\subset\mathbb{C}\setminus\mathbb{D}$ with connected complement and for every 
pair of functions $(g_1,g_2)\in A(K)\times A(K),$ there exists a subsequence of positive integers 
$(\mu_n)$ such that 
$$\sup_{z\in K}\vert S_{\mu_n}(f)(z)-g_1(z)\vert\rightarrow 0\hbox{ and }
\sup_{z\in K}\vert S_{\lambda_{\mu_n}}(f)(z)-g_2(z)\vert\rightarrow 0,\hbox{ as }
n\rightarrow +\infty.$$
Such a function will be called {\it doubly universal Taylor series with respect to the 
sequences} $(n),$ $(\lambda_n).$ }
\end{definition}
Using tools from potential theory they proved that the set $\mathcal{U}(\mathbb{D},(\lambda_n),0)$ 
is non-empty if and only if $\limsup_n\frac{\lambda_n}{n}=+\infty$. Moreover they obtained that the existence of a doubly universal series implies topological genericity of such series. 
In the present paper we show that 
the advanced knowledge of potential theory does not play a dominant role to obtain the proof of the implication 
$\mathcal{U}(\mathbb{D},(\lambda_n),0)\ne\emptyset\Rightarrow \limsup_n\frac{\lambda_n}{n}=+\infty$. 
Instead we employ some polynomial inequalities which were recently used to study the densities of approximation subsequences of universal Taylor series in the sense of Nestoridis (see \cite{MouMu1,MouMu2}). It seems quite natural that the arithmetic structure of 
subsequences along which the partial sums possess the universal approximation property is connected with 
the above notion of disjointness. As a consequence, we obtain that the set of doubly universal Taylor series is densely lineable, i.e. contains a dense vector subspace except $0.$ This concept gives some information about the algebraic structure of the set of such series. Several authors were 
recently interested in this phenomenon (see for instance \cite{bernal}). Further, since we avoid 
the use of potential theory in a large way, we extend the aforementioned Costakis-Tsirivas result to the case of the sequence of partial sums of Taylor development at $0$ of infinitely differentiable 
functions on $\mathbb{R}.$ This generalization uses in an essential way 
classical Bernstein polynomials of given continuous functions on intervals of the type $[0,A].$ In particular, these specific polynomials possess a useful property 
in our context: we control both their degree and their valuation provided that the associated function vanishes on a neighborhood of zero. Finally we return to the 
connection between doubly universality and topological recurrence. In a recent note, Costakis and 
Parissis proved that every frequently Ces\`aro operators is topologically multiply recurrent \cite{CP}.  
In our context, we show that the Ces\`aro means of partial sums of a real or complex power series cannot 
be frequently universal series. So the doubly universality, which is related to the topological multiply recurrence, does not imply the frequent Ces\`aro universality.\\ 
The paper is organized as follows. In Section \ref{S_d_CT} we give a new shorter proof of the implication $\limsup_n\frac{\lambda_n}{n}<+\infty\Rightarrow\mathcal{U}(\mathbb{D},(\lambda_n),0) = \emptyset$ and we establish the algebraic genericity of the set $\mathcal{U}(\mathbb{D},(\lambda_n),0).$  In Section \ref{S_d_infinitely}, we are interested in the case of doubly universal infinitely differentiable functions with respect to an increasing sequence $(\lambda_n)$ of positive integers. We establish both the topological and algebraic genericity of the set of such functions provided that 
$\limsup_n\frac{\lambda_n}{n}=+\infty$ again. In Section \ref{section_remark}, we study the frequent 
Ces\`aro universal series and finally we give an example, in a different context, of the existence of doubly universal series with respect to an increasing sequence $(\lambda_n)
\ne \mathbb{N}$ of positive integers without additional assumption.

\section{Doubly universal Taylor series in the complex plane} \label{S_d_CT} 
In this section, we begin by giving a proof with rather elementary arguments of the fact that 
$\limsup_n\frac{\lambda_n}{n}<+\infty$ implies that $\mathcal{U}(\mathbb{D},(\lambda_n),0) = \emptyset.$ 
To do this, let us recall the nice Tur\'an inequality \cite{Tur}, which estimates the global behavior of a polynomial on a circle $\{z\in\mathbb{C}:\vert z\vert =r\}$ by its supremum on subsets of 
$\{z\in\mathbb{C}:\vert z\vert =r\}$. 

\begin{lemma}\label{TuranIneq} Let $Q$ be a polynomial of arbitrary degree 
which possesses only $n$ non zero coefficients. Then for any $r>0$ and any $\delta$ 
($0<\delta<2\pi$) 
$$\sup_{\vert z\vert =r}\vert Q(z)\vert\leq \left(\frac{4\pi e}{\delta}\right)^n\sup_{\vert t\vert\leq \delta/2}
\vert Q(re^{it})\vert.$$
\end{lemma}
\noindent For $r>0$ and $0<\delta<2\pi,$ $\Gamma_{r,\delta}$ will be the set 
$$\Gamma_{r,\delta}=\left\{z\in\mathbb{C};\ \vert z\vert=r\hbox{ and } -\frac{\delta}{2}\leq \arg(z)\leq \frac{\delta}{2}\right\}$$ 
and $C_\delta=\frac{4\pi e}{\delta}$ the constant of the above Tur\'an inequality. 

Now we state \cite[Proposition 4.5]{CT} and we furnish a simple proof. 

\begin{proposition}\label{propCT} Let $(\lambda_n)$ be a strictly increasing sequence of positive integers. 
Assume that 
$\limsup_n\left(\frac{\lambda_n}{n}\right)<+\infty.$ Then the set $\mathcal{U}(\mathbb{D},(\lambda_n),0)$ is empty. 
\end{proposition}

\begin{proof} 
The proof is based on the use of Tur\'an's inequality. We argue by contradiction. Take $f=\sum_{n\geq 0}a_nz^n$ in 
$\mathcal{U}(\mathbb{D},(\lambda_n),0).$ Since we have $\limsup_n\left(\frac{\lambda_n}{n}\right)<+\infty,$ 
there exists $d>0$ such that 
\begin{equation}\label{equ_dens_1}\forall n\in\mathbb{N},\ \lambda _n\leq dn.
\end{equation}
Let $r>0$ and $0<\delta<2\pi.$ Fix a compact set $K\subset\mathbb{C}\setminus\mathbb{D}$ with connected 
complement. Let us choose $R>0$ so that 
\begin{equation}\label{hypalpha}
\frac{R}{C_{\delta}^{d}}>\sup_{z\in K}\vert z\vert.
\end{equation}
Clearly the set $K_{R,\delta}:=\Gamma_{R,\delta}\cup K$ is a compact set with connected complement. Since 
$f\in \mathcal{U}(\mathbb{D},(\lambda_n),0)$ there exists an increasing $(\mu_n)$ of positive integers 
such that 
$$\sup_{z\in K_{r,\delta}}\vert S_{\mu_n}(f)(z)-1\vert\rightarrow 0\hbox{ and }
\sup_{z\in K_{r,\delta}}\vert S_{\lambda_{\mu_n}}(f)(z)\vert\rightarrow 0,\hbox{ as }n\rightarrow +\infty.$$
Therefore, for any $0<\varepsilon<1,$ we can find $n_0\in\mathbb{N}$ such that 
for all $n\geq n_0,$ 
\begin{equation}\label{double-estim}\sup_{z\in K_{r,\delta}}\vert S_{\mu_n}(f)(z)-1\vert<\varepsilon/4\hbox{ and }
\sup_{z\in K_{r,\delta}}\vert S_{\lambda_{\mu_n}}(f)(z)\vert<\varepsilon/4.
\end{equation} 
In particular, we have, for every $n\geq n_0,$
$$\sup_{z\in \Gamma_{R,\delta}}\left\vert 
\sum_{j=0}^{\lambda_{\mu_n}}
a_jz^{j}\right\vert<\varepsilon/4.$$
Using Lemma \ref{TuranIneq} and Cauchy estimates, we get for every $n\geq n_0$ and 
$j=0,\dots, \lambda_{\mu_n},$ 
\begin{equation}\label{equafond0}\vert a_{j}\vert^{1/j}  \leq \frac{\varepsilon^{1/j}}{4^{1/j}R}
C_{\delta}^{\lambda_{\mu_{n}}/{j}}.
\end{equation} 
Taking into account (\ref{equ_dens_1}), we get for every $n\geq n_0$ and 
$j=1+\mu_n,\dots, \lambda_{\mu_n},$ 
\begin{equation}\label{equafond}\vert a_{j}\vert^{1/j}  \leq \frac{\varepsilon^{1/j}}{4^{1/j}}\frac{C_{\delta}^{d}}{R}.
\end{equation} 

By (\ref{equafond}) and (\ref{hypalpha}), we deduce that there exists a positive integer 
$n_2\geq n_0$ such that for $n\geq n_2$ the following estimate holds
$$\sup_{z\in K}\vert\sum_{j=1+\mu_n}^{\lambda_{\mu_n}}a_jz^j\vert<\varepsilon/4.$$
Finally using the inequality $\sup_{z\in K}\vert S_{\lambda_{\mu_n}}(f)(z)\vert<\varepsilon/4,$ 
we have, for all $n\geq n_2,$ 
\begin{equation}\label{festim}\sup_{z\in K}\vert S_{{\mu_n}}(f)(z)\vert=\sup_{z\in K}\vert S_{\lambda_{\mu_n}}(f)(z)-
\sum_{j=1+\mu_n}^{\lambda_{\mu_n}}a_jz^j\vert
\leq \sup_{z\in K}\vert S_{\lambda_{\mu_n}}(f)(z)\vert+\sup_{z\in K}\vert 
\sum_{j=1+\mu_n}^{\lambda_{\mu_n}}a_jz^j\vert<\varepsilon/2.
\end{equation}
Combining (\ref{double-estim}) with (\ref{festim}) we obtain 
$$1\leq \sup_{z\in K}\vert S_{{\mu_n}}(f)(z)-1\vert+\sup_{z\in K}\vert S_{{\mu_n}}(f)(z)\vert\leq  3\varepsilon/4,$$
which is a contradiction. This completes the proof of the proposition.
\end{proof}

Further we are interested in the algebraic structure of $\mathcal{U}(\mathbb{D},(\lambda_n),0).$ First let us define the set of doubly universal Taylor series along a given subsequence. 

\begin{definition}\label{def_CT2}{\rm Let $(\lambda_n)$ and $\mu=(\mu_n)$ be increasing sequences of positive integers. 
A function $f\in H(\mathbb{D})$ belongs to the class $\mathcal{U}^{(\mu)}(\mathbb{D},(\lambda_n),0)$ if 
for every compact set $K\subset\mathbb{C}\setminus\mathbb{D}$ with connected complement and for every 
pair of functions $(g_1,g_2)\in A(K)\times A(K),$ there exists a subsequence of positive integers 
$(\nu_n)\subset \mu$ such that 
$$\sup_{z\in K}\vert S_{\nu_n}(f)(z)-g_1(z)\vert\rightarrow 0\hbox{ and }
\sup_{z\in K}\vert S_{\lambda_{\nu_n}}(f)(z)-g_2(z)\vert\rightarrow 0,\hbox{ as }
n\rightarrow +\infty.$$}
\end{definition}

\begin{remark}\label{rem_CT} {\rm Arguing as in the proof of Proposition \ref{propCT}, we obtain that the existence of universal elements in 
$\mathcal{U}^{(\mu)}(\mathbb{D},(\lambda_n),0)$ implies $\limsup_n\frac{\lambda_{\mu_n}}{\mu_n}=+\infty.$ 
On the other hand, the hypothesis $\limsup_n\frac{\lambda_{\mu_n}}{\mu_n}=+\infty$ implies that the 
set $\mathcal{U}^{(\mu)}(\mathbb{D},(\lambda_n),0)$ is $G_\delta$ and dense in $H(\mathbb{D})$. The proof works as in \cite[Proposition 4.1]{CT} with 
obvious modifications.}
\end{remark} 
Moreover a careful examination of the proof of Proposition \ref{propCT} gives the 
following lemma.

\begin{lemma}\label{lemma_double} Let $(\lambda_n)$ and $\mu=(\mu_n)$ be increasing sequences of positive integers. 
Let $f$ be in $\mathcal{U}^{(\mu)}(\mathbb{D},(\lambda_n),0).$ For every compact set $K\subset\mathbb{C}\setminus\mathbb{D},$ 
with $K^c$ connected, and for every pair of functions $(g_1,g_2)\in A(K)\times A(K),$ with $g_1\ne g_2,$ 
there exists a subsequence $(\nu_n)$ of $\mu$
with $\limsup_n\frac{\lambda_{\nu_n}}{\nu_n}=+\infty$ such that 
$$\sup_{z\in K}\vert S_{\nu_n}(f)(z)-g_1(z)\vert\rightarrow 0\hbox{ and }
\sup_{z\in K}\vert S_{\lambda_{\nu_n}}(f)(z)-g_2(z)\vert\rightarrow 0,\hbox{ as }
n\rightarrow +\infty.$$
\end{lemma}

\begin{proof} As in the proof of Proposition \ref{propCT}, let us choose $R>0$ so that 
$\frac{R}{C_{\delta}^{d}}>\sup_{z\in K}\vert z\vert$ and consider the compact set 
$K_{R,\delta}:=\Gamma_{R,\delta}\cup K.$ Since 
$f\in \mathcal{U}^{(\mu)}(\mathbb{D},(\lambda_n),0)$ there exists an increasing $(\nu_n)\subset\mu$ of positive integers 
such that 
$$\sup_{z\in K_{r,\delta}}\vert S_{\nu_n}(f)(z)-g_1(z)\vert\rightarrow 0\hbox{ and }
\sup_{z\in K_{r,\delta}}\vert S_{\lambda_{\nu_n}}(f)(z)-g_2(z)\vert\rightarrow 0,\hbox{ as }n\rightarrow +\infty.$$ 
Arguing as in the end of the proof of Proposition \ref{propCT}, we deduce that we have necessary  
$\limsup_n\frac{\lambda_{\nu_n}}{\nu_n}=+\infty.$
\end{proof}

Combining Remark \ref{rem_CT} with Lemma \ref{lemma_double} we get that the set $\mathcal{U}(\mathbb{D},(\lambda_n),0)\cup\{0\}$ 
is algebraically generic. 

\begin{theorem}\label{sub_vec} Let $(\lambda_n)$ be a strictly increasing sequence of positive integers such that 
$\limsup_n\left(\frac{\lambda_n}{n}\right)=+\infty.$ The set $\mathcal{U}(\mathbb{D},(\lambda_n),0)\cup\{0\}$ contains a dense vector subspace 
of $H(\mathbb{D}).$
\end{theorem}

\begin{proof}  
We proceed as in the proof of \cite[Theorem 3]{bgnp} with essential modifications. 
Let us fix a dense sequence $(h_l)_l$ in $H(\mathbb{D}).$ In the following, $d_{H(\mathbb{D})}$ denotes the 
standard metric of $H(\mathbb{D}).$ Let $(K_m)$ be a family of compact sets with connected 
complement and $K_m\cap\mathbb{D}=\emptyset$ for every $m\in\mathbb{N}$ such that 
every compact subset $K\subset\{z\in\mathbb{C}:\vert z\vert\geq 1\},$ with 
$K^c$ connected, is contained in some $K_n,$ $n\in \mathbb{N}$ \cite[Lemma 2.1]{Nes}. 
We construct a sequence 
$(f_{l})_l$ in $H(\mathbb{D})$ and sequences $\mu^{k,l}$ of positive integers 
satisfying the following conditions, for any $k,l\geq 1,$
\begin{itemize}
\item $\mu^{k,l}$ is a subsequence of $\mu^{k,l-1}$, with $\mu^{k,0}=\mathbb{N},$
\item $d_{H(\mathbb{D})}(f_l,h_l)<2^{-l},$
\item $\limsup_n\frac{\lambda_{\mu_n^{k,l}}}{\mu_n^{k,l}}=+\infty,$
\item $f_l$ belongs to $\bigcap_{k\geq 1}\mathcal{U}^{(\mu^{k,l-1})}(\mathbb{D}, (\lambda_n),0),$
\item $\sup_{z\in K_k}\vert S_{\lambda_{\mu_n^{k,l}}}(f_l)\vert\rightarrow 0,$ 
$\sup_{z\in K_k}\vert S_{\mu_n^{k,l}}(f_l)-1\vert\rightarrow 0,$ and 
$d_{H(\mathbb{D})}(S_{\mu_n^{k,l}}(f_l,\zeta),f_l)\rightarrow 0,$ as $n\rightarrow +\infty.$ 
\end{itemize} 
To do this, observe that first we can choose $f_1$ in the dense set $\mathcal{U}(\mathbb{D},(\lambda_n),0)$ 
so that $d_{H(\mathbb{D})}(f_1,h_1)<2^{-1}.$ Therefore, applying Lemma \ref{lemma_double}, for any 
$k\geq 1,$ one may find a subsequence $\mu^{k,1}$ with $\limsup_n\frac{\lambda_{\mu_n^{k,1}}}{\mu_n^{k,1}}=+\infty$, 
such that  $\sup_{K_k}\vert S_{\lambda_{\mu_n^{k,1}}}(f_1)\vert\rightarrow 0$ and 
$\sup_{K_k}\vert S_{\mu_n^{k,1}}(f_1)-1\vert\rightarrow 0$ as $n\rightarrow +\infty.$ 

At step 2, we choose 
$f_2 \in \bigcap_{k\geq 1}\mathcal{U}^{(\mu^{k,1})}(\mathbb{D}, (\lambda_n),0),$ which is a $G_\delta$ and dense subset of 
$H(\mathbb{D}),$ with $d_{H(\mathbb{D})}(f_2,h_2)<2^{-2}.$ In particular, according to Lemma \ref{lemma_double} 
for any $k\geq 1,$ there exists a subsequence $\mu^{k,2}$ of 
$\mu^{k,1}$, with $\limsup_n\frac{\lambda_{\mu_n^{k,2}}}{\mu_n^{k,2}}=+\infty$ such that 
$\sup_{K_k}\vert S_{\lambda_{\mu_n^{k,2}}}(f_2)\vert\rightarrow 0$ and $\sup_{K_k}\vert S_{\mu_n^{k,2}}(f_2)-1\vert\rightarrow 0$ 
as $n\rightarrow +\infty.$ Then we repeat  
the same arguments to construct 
$(f_{l})_l$ in $H(\mathbb{D})$ and sequences $\mu^{k,l}$ of positive integers satisfying the above properties. To finish the proof, it is sufficient to check that the linear span of the $(f_l)$ is both dense in $H(\mathbb{D})$ and contained in 
$\mathcal{U}(\mathbb{D},(\lambda_n),0)$, except for the zero function. The density is clear by construction. Moreover let $f=\alpha_1f_1+\dots +\alpha_mf_m,$ with $\alpha_m\ne 0.$ Let 
us consider two polynomials $g_1,$ $g_2$ and a compact set $K$ with connected complement and 
$K\cap\Omega=\emptyset.$ There exists $k$ such that $K\subset K_k.$ 
Since $f_m\in \mathcal{U}^{(\mu^{k,m-1})}(\mathbb{D},(\lambda_n),0),$ there exists a sequence $(\gamma_n)$ of positive integers with 
$(\gamma_n)\subset \mu^{k,m-1}$ such that 
\begin{equation}\label{ds1}\sup_{z\in K_k}\vert S_{\lambda_{\gamma_n}}(\alpha_m f_m)(z)-g_1(z)\vert\rightarrow 0
\hbox{ and }\sup_{z\in K_k}\vert S_{\gamma_n}(\alpha_m f_m)(z)-(g_2(z)-\sum_{i=1}^{m-1}\alpha_i)\vert\rightarrow 0\hbox{ as }
n\rightarrow +\infty.
\end{equation}
Observe that $(\gamma_n)$ is a subsequence of any $\mu^{k,l}$ for $l\leq m-1.$ Hence by construction we have, for any $l\leq m-1,$ 
\begin{equation}\label{ds2}
\sup_{z\in K_k}\vert S_{\lambda_{\gamma_n}}(\alpha_l f_l)(z)\vert\rightarrow 0\hbox{ and }
\sup_{z\in K_k}\vert S_{\gamma_n}(\alpha_l f_l)(z)-\alpha_l\vert\rightarrow 0\hbox{ as }
n\rightarrow +\infty.
\end{equation}
Finally from (\ref{ds1}) and (\ref{ds2}) we get 
$$\sup_{z\in K_k}\vert S_{\lambda_{\gamma_n}}(f)(z)-g_1(z)\vert\rightarrow 0$$
and 
$$\sup_{z\in K_k}\vert S_{\gamma_n}(f)(z)-g_2(z)\vert \leq 
\sup_{z\in K_k}\vert S_{\gamma_n}(\alpha_mf_m)(z)-(g_2(z)-\sum_{i=1}^{m-1}\alpha_i)\vert +
\sum_{i=1}^{m-1}\sup_{z\in K_k}\vert S_{\gamma_n}(\alpha_if_i)(z)-\alpha_i\vert \rightarrow 0$$
as $n\rightarrow +\infty,$
which implies that $f$ belongs to 
$\mathcal{U}(\mathbb{D},(\lambda_n),0).$ 
\end{proof}

\section{Doubly universal infinitely differentiable functions}\label{S_d_infinitely} First let us introduce some notations and terminology. 
We consider the set $C_0^{\infty}(\mathbb{R})$ of functions $f\in C^{\infty}(\mathbb{R})$ with $f(0)=0.$ Its topology is defined by the seminorms $\sup_{ x\in [-m,m]}\vert f^{(j)}(x)\vert,$ $j,m\in \mathbb{N}$ and the associated standard translation-invariant metric will be denoted by $d_{C_0^{\infty}(\mathbb{R})}$. Moreover we will consider the classical space 
$\mathbb{R}^{\mathbb{N}}$ endowed with the metric $d_{\mathbb{R}^{\mathbb{N}}}$  defined 
by $d_{\mathbb{R}^{\mathbb{N}}}((u_n),(v_n))=\sum_{n\geq 0}2^{-n}(\max_{0\leq j\leq n}\vert u_j-v_j\vert/
(1+\max_{0\leq j\leq n}\vert u_j-v_j\vert)).$ 
The metric space $(\mathbb{R}^{\mathbb{N}},d_{\mathbb{R}^{\mathbb{N}}})$ is complete. \\
As far as we know Fekete exhibited the first example of universal series by showing that 
there exists a formal power series $\sum_{n\geq 1}a_nx^n$ with the following property: for every continuous function $g$ on $[-1,1]$ with $g(0)=0$ there exists an increasing 
sequence $(\lambda_n)$ of positive integers such that $\sup_{x\in [-1,1]}\vert\sum_{k=1}^{\lambda_n}a_kx^k-g(x)\vert\rightarrow 0,$ as 
$n\rightarrow +\infty$ \cite{Pal}. A slight modification of Fekete's proof combined with Borel's theorem allows to obtain $C^{\infty}$-function whose partial sums of 
its Taylor series around $0$ approximate every continuous functions vanishing at $0$ locally uniformly in $\mathbb{R}$ (see \cite{ge}). In the present section, we are going to obtain a natural extension of the results of Section \ref{S_d_CT} to the 
case of Fekete functions, exploiting the fact that we did not need to use advanced tools of potential theory to study the class of doubly (complex) universal Taylor series.  
 
First of all, let us mention a useful inequality for polynomials in many variables between the complex and 
the real sup-norms \cite{AronBeau,Klimek}. 
\begin{theorem} \label{thmAronBeau}
There exists a constant $C>1$ such that, for any polynomial $P$ of degree $n$ in $k$ variables 
with real coefficients, we have  
$$\sup_{\vert z_1\vert=\dots=\vert z_k\vert=1}\vert P(z_1,\dots,z_k)\vert 
\leq C^n \sup_{x_1,\dots,x_k\in [-1,1]}\vert P(x_1,\dots,x_k)\vert.$$
\end{theorem}
\noindent In Theorem \ref{thmAronBeau}, we can choose the constant $C$ to be $1+\sqrt{2}$ \cite{AronBeau,Klimek}. 

\begin{definition}{\rm  Let $(\lambda_n)$ be a strictly increasing sequence of positive integers. A function 
$f\in  C_0^{\infty}(\mathbb{R})$ belongs to the class $\mathcal{U}(C_0^{\infty}(\mathbb{R}),(\lambda_n))$ if for every compact set $K\subset\mathbb{R}$ and for every pair $(h_1,h_2)$ of continuous functions $h_1,h_2:\mathbb{R}\rightarrow\mathbb{R}$ vanishing at 
zero, there exists a subsequence of positive integers $(\mu_n)$ such that 
$$\sup_{x\in K}\left\vert \sum_{k=0}^{\lambda_{\mu_n}}\frac{f^{(k)}(0)}{k!}x^k-h_1(x)\right\vert\rightarrow 0\hbox{ and }
\sup_{x\in K}\left\vert \sum_{k=0}^{\mu_n}\frac{f^{(k)}(0)}{k!}x^k-h_2(x)\right\vert\rightarrow 0\hbox{ as }n\rightarrow +\infty.$$}
\end{definition}
Now we can state a version of Proposition \ref{propCT} in this context.

\begin{proposition}\label{direct_A0} Let $(\lambda_n)$ be a strictly increasing sequence of positive integers. 
Assume that $\limsup_n\left(\frac{\lambda_n}{n}\right)<+\infty.$ Then the set 
$\mathcal{U}(C_0^{\infty}(\mathbb{R}),(\lambda_n))$ is empty. 
\end{proposition}

\begin{proof} We argue by contradiction. Take $f$ in 
$\mathcal{U}(C_0^{\infty}(\mathbb{R}),(\lambda_n))$ and set, for every $k\geq 0,$ $a_k=\frac{f^{(k)}(0)}{k!}.$ 
Since we have $\limsup_n\left(\frac{\lambda_n}{n}\right)<+\infty,$ 
there exists $d>0$ such that 
\begin{equation}\label{estim_dens}
\forall n\in\mathbb{N},\ \lambda_n\leq dn.
\end{equation} 
Let us fix 
\begin{equation}\label{choice_epsilon}
0<\varepsilon <\frac{1}{2C^d},
\end{equation} 
where $C$ is the absolute constant given by 
Theorem \ref{thmAronBeau}. 
Since $f$ belongs to $\mathcal{U}(C_0^{\infty}(\mathbb{R}),(\lambda_n)),$ there exists a subsequence $(\mu_n)$ of positive integers such that, for any $n\geq n_0,$ 
\begin{equation}\label{estim_epsilon_1}
\sup_{x\in [-1,1]}\vert \sum_{j=1}^{\mu_n}a_jx^j-x\vert <\varepsilon/4\hbox{ and }
\sup_{x\in [-1,1]}\vert \sum_{j=1}^{\lambda_{\mu_n}}a_j x^j\vert <\varepsilon/4.
\end{equation}
Using Theorem \ref{thmAronBeau} we get, for every $n\geq n_0,$
\begin{equation}\label{equ_disc}\sup_{\vert z\vert = 1}\left\vert 
\sum_{j=1}^{\lambda_{\mu_n}}
a_j z^{j}\right\vert\leq C^{\lambda_{\mu_n}}\frac{\varepsilon}{4}.
\end{equation} 
It follows from Cauchy's formula, for every $n\geq n_0$ and 
$j=0,\dots, \lambda_{\mu_n},$ 
\begin{equation}\label{equafond_10}\vert a_{j}\vert^{1/j}  \leq \frac{\varepsilon^{1/j}}{4^{1/j}}
C^{\lambda_{\mu_{n}}/{j}}.
\end{equation}
From (\ref{estim_dens}) we deduce, for $j=1+\mu_n,\dots,\lambda_{\mu_n},$
\begin{equation}\label{equafond_10_11}\vert a_{j}\vert^{1/j}  \leq \frac{\varepsilon^{1/j}}{4^{1/j}}
C^{d},
\end{equation}
and therefore we get, for $n\geq n_0,$ 
$$\sup_{x\in [-1/(2C^d),1/(2C^d)]}\vert\sum_{j=1+\mu_n}^{\lambda_n}a_jx^j\vert\leq 
\frac{\varepsilon}{4}\sum_{j=1+\mu_n}^{\lambda_n} \frac{1}{2^j}\leq 
\frac{\varepsilon}{4}.$$
Finally using (\ref{estim_epsilon_1}) we have, for all $n\geq n_0,$ 
\begin{equation}\label{festim2_1}\sup_{x\in [-1/(2C^d),1/(2C^d)]}\vert \sum_{j=1}^{\mu_n}a_j x^j\vert
\leq \sup_{x\in [-1/(2C^d),1/(2C^d)]}\vert \sum_{j=1}^{\lambda_{\mu_n}}a_j x^j\vert+\sup_{x\in [-1/(2C^d),1/(2C^d)]}\vert 
\sum_{j=1+\mu_n}^{\lambda_{\mu_n}}a_jx^j\vert<\varepsilon/2.
\end{equation}
Combining (\ref{estim_epsilon_1}) with (\ref{festim2_1}) we obtain, for $n\geq n_0,$ 
$$\begin{array}{rcl}\displaystyle\frac{1}{2C^d}=\sup_{x\in [-1/(2C^d),1/(2C^d)]}\vert x\vert&\leq &
\displaystyle\sup_{x\in [-1/(2C^d),1/(2C^d)]}\vert 
\sum_{j=1}^{\mu_n}a_j x^j\vert + 
\sup_{x\in [-1/(2C^d),1/(2C^d)]}\vert \sum_{j=1}^{\mu_n}a_j x^j-x\vert\\
&\leq&\displaystyle \frac{\varepsilon}{2}+
\frac{\varepsilon}{4}<\varepsilon.\end{array}$$
This last inequality gives a contradiction with (\ref{choice_epsilon}). This completes the proof.
\end{proof}

To obtain the converse result, we will follow the main ideas of the proof of \cite[Proposition 4.1]{CT}. 
First we need a quantitative approximation polynomial lemma which will play the role of \cite[Theorem 2.1]{CT}. We have to approximate a 
given continuous function vanishing at $0$ by polynomials whose both degrees and valuations are imposed. Exploiting the form of the classical Bernstein polynomials, 
we begin with the case where the approximation takes place on compact subsets of $[0,+\infty).$  

\begin{lemma}\label{Bernstein_lemma} Let $(l_n)$ and $(m_n)$ be two strictly increasing sequences of positive integers such that 
$l_n\leq m_n$ and $\frac{m_n}{l_n}\rightarrow +\infty$ as $n\rightarrow +\infty.$ Let $A>0.$ For every continuous function $h:\mathbb{R_+}\rightarrow \mathbb{R},$ with $h(0)=0,$ there exists 
a sequence $(P_n)$ of real polynomials of the form $P_n(x)=\sum_{k=l_n}^{m_n}c_{n,k}x^k,$ such that
$$\sup_{x\in [0,A]}\vert P_n(x)-h(x)\vert\rightarrow 0,\hbox{ as }n\rightarrow +\infty.$$
\end{lemma}

\begin{proof} Let $\varepsilon>0.$ 
By continuity of the function $h$ at $0,$ one can find $\eta>0$ such that for all $x\in [0,\eta),$ 
$\vert h(x)\vert <\varepsilon/4.$ Let us consider the continuous function $\tilde{h}$ defined on $\mathbb{R}_+$ by
$$\tilde{h}(x)=\left\{\begin{array}{ll}0&\hbox{ for }0\leq x\leq \eta/2\\
\frac{2 h(\eta)}{\eta}(x-\eta/2)&\hbox{ for }\eta/2\leq x\leq \eta\\
h(x)&\hbox{ for }x\geq \eta\end{array}\right.$$ 
Then, for every $n\geq 1,$ let us consider 
$B_{m_n}(\tilde{h})$ its Bernstein polynomial of degree $m_n,$ given by 
$$B_{m_n}(\tilde{h})(x)=\frac{1}{A^{m_n}}\sum_{k=0}^{m_n}{m_n\choose k}\tilde{h}\left(A\frac{k}{m_n}\right)
x^k(A-x)^{m_n-k}.$$
Moreover since the sequence $(l_n/m_n)$ converges to $0,$ 
there exists a positive integer $N_1$ such that, for every $n\geq N_1,$ $A l_n/m_n\leq \eta/2.$ Therefore by construction, for every $n\geq N_1$ and for $k= 0,1,\dots,l_n,$we have $\tilde{h}\left(A\frac{k}{m_n}\right)=0.$ The Bernstein 
polynomials of $\tilde{h}$ have the following form 
$$B_{m_n}(\tilde{h})(x)=\frac{1}{A^{m_n}}\sum_{k=l_n}^{m_n}{m_n\choose k}\tilde{h}\left(A\frac{k}{m_n}\right)
x^k(A-x)^{m_n-k}:=\sum_{k=l_n}^{m_n}c_{n,k}x^k.$$
Obviously the function $\tilde{h}$ is continuous on $[0,A].$ So it is known that the sequence 
$(B_{n}(\tilde{h}))$ converges uniformly to $\tilde{h}$ on $[0,A]$ \cite{Bernstein}. Thus we deduce the existence of a positive integer $N_2>N_1$ such that, for every $n\geq N_2,$ 
$$\sup_{x\in [0,A]}\vert B_{m_n}(\tilde{h})(x)-\tilde{h}(x)\vert < \varepsilon/2.$$
Now, for $x \leq \eta/2,$ since $\tilde{h}(x)=0,$ the triangle inequality gives 
$$\vert B_{m_n}(\tilde{h})(x)-h(x)\vert \leq \vert B_{m_n}(\tilde{h})(x)\vert +\vert h(x)\vert <
\varepsilon/2 +\varepsilon/4 < \varepsilon.$$
On the other hand, for $\eta/2\leq x\leq \eta,$ we have 
$$\vert B_{m_n}(\tilde{h})(x)-h(x)\vert \leq \vert B_{m_n}(\tilde{h})(x)-\tilde{h}(x)\vert +
\vert \tilde{h}(x)-h(x)\vert <\varepsilon/2 +\varepsilon/2=\varepsilon.$$
Finally for $x\geq \eta,$ we have $\tilde{h}(x)=h(x)$ and we get $\vert B_{m_n}(\tilde{h})(x)-h(x)\vert <\varepsilon.$
This completes the proof.
\end{proof}

Then we extend Lemma \ref{Bernstein_lemma} to the case of symmetric intervals $[-A,A].$ 

\begin{lemma}\label{Bernstein_lemma2} Let $(l_n)$ and $(m_n)$ be two strictly increasing sequences of positive integers such that 
$l_n\leq m_n$ and $\frac{m_n}{l_n}\rightarrow +\infty$ as $n\rightarrow +\infty.$ Let $A>0.$ For every continuous function $h:\mathbb{R}\rightarrow \mathbb{R},$ with $h(0)=0,$ there exists 
a sequence $(P_n)$ of real polynomials of the form $P_n(x)=\sum_{k=l_n}^{m_n}c_{n,k}x^k,$ such that
$$\sup_{x\in [-A,A]}\vert P_n(x)-h(x)\vert\rightarrow 0,\hbox{ as }n\rightarrow +\infty.$$
\end{lemma}

\begin{proof} Let $\varepsilon>0.$ 
By continuity of the function $h$ at $0,$ one can find $\eta>0$ such that for all $x\in [0,\eta),$ 
$\vert h(x)\vert <\varepsilon/4.$ Let us consider the continuous function $\tilde{h}$ defined on $\mathbb{R}$ by
$$\tilde{h}(x)=\left\{\begin{array}{ll}0&\hbox{ for }\vert x\vert \leq \eta/2\\
\frac{2 h(\eta)}{\eta}(x-\eta/2)&\hbox{ for }\eta/2\leq x\leq \eta\\
\frac{-2 h(-\eta)}{\eta}(x+\eta/2)&\hbox{ for }-\eta\leq x\leq -\eta/2\\
h(x)&\hbox{ for }\vert x\vert \geq \eta\end{array}\right.$$ 
Observe that we have 
\begin{equation}\label{first_htilde}
\sup_{x\in [-A,A]}\vert h(x)-\tilde{h}(x)\vert \leq \varepsilon/2.
\end{equation}
Define also the continuous function $g:\mathbb{R}\rightarrow \mathbb{R},$ 
$x\mapsto \tilde{h}(x)/x^2.$ By classical Weierstrass approximation theorem one can find a polynomial $P$ such that 
$$\sup_{x\in [-A,A]}\vert P(x)-g(x)\vert <\varepsilon/4A^2.$$ 
We deduce the following inequality 
\begin{equation}\label{second_htilde}
\sup_{x\in [-A,A]}\vert x^2P(x)-\tilde{h}(x)\vert <\varepsilon/4.
\end{equation} 
Set $W(x):=x^2P(x).$ Observe that $W(0)=W'(0)=0.$ Let us write 
$$W(x)=Q_1(x^2)+xQ_2(x^2)$$
where $Q_1$ and $Q_2$ are polynomials vanishing at $0.$ Then we apply Lemma \ref{Bernstein_lemma} to find two sequences 
of polynomials $P_{n,1}$ and $P_{n,2}$ of the form 
$$P_{n,1}(x)=\sum_{k=\lfloor l_n/2\rfloor+1}^{\lfloor m_n/2\rfloor}c_{n,k}^{(1)}x^k\hbox{ and }
P_{n,2}(x)=\sum_{k=\lfloor l_n/2\rfloor}^{\lfloor (m_n-1)/2\rfloor}c_{n,k}^{(2)}x^k$$ such that 
$$\sup_{x\in [0,A^2]}\vert P_{n,1}(x)-Q_1(x)\vert\rightarrow 0\hbox{ and }
\sup_{x\in [0,A^2]}\vert P_{n,2}(x)-Q_2(x)\vert\rightarrow 0,\hbox{ as }n\rightarrow +\infty.$$
For $n$ large enough we get
$$\sup_{x\in [-A,A]}\vert P_{n,1}(x^2)-Q_1(x^2)\vert <\varepsilon/8\hbox{ and }
\sup_{x\in [-A,A]}\vert P_{n,2}(x^2)-Q_2(x^2)\vert <\varepsilon/8A.$$
Thus by construction the polynomial $\tilde{W}(x):=P_{n,1}(x^2)+xP_{n,2}(x^2)$ has the following form 
\begin{equation}\label{formwtilde}
\tilde{W}(x)=\sum_{k=l_n}^{m_n}c_{n,k}x^k
\end{equation} 
and we have 
\begin{equation}\label{three_htilde}
\begin{array}{rcl}\displaystyle\sup_{x\in [-A,A]}\vert \tilde{W}(x)-W(x)\vert&\leq& 
\displaystyle\sup_{x\in [-A,A]}\vert P_{n,1}(x^2)-Q_1(x^2)\vert + \sup_{x\in [-A,A]}\vert x(P_{n,2}(x^2)-Q_2(x^2))\vert\\&
<&\displaystyle\frac{\varepsilon}{8}+A\frac{\varepsilon}{8A}=\frac{\varepsilon}{4}.\end{array}
\end{equation}
Finally combining the triangle inequality with (\ref{first_htilde}), (\ref{second_htilde}) and (\ref{three_htilde}), 
we get
\begin{equation}\label{estimatewtilde}
\begin{array}{rcl}\displaystyle\sup_{[-A,A]}\vert h-\tilde{W}\vert &\leq &\displaystyle
\sup_{[-A,A]}\vert h-\tilde{h}\vert + \sup_{[-A,A]}\vert \tilde{h}-W\vert + 
\sup_{[-A,A]}\vert W-\tilde{W}\vert \\&<&\displaystyle \frac{\varepsilon}{2}+\frac{\varepsilon}{4}+\frac{\varepsilon}{4}
=\varepsilon.\end{array}
\end{equation}
Thus the polynomial $\tilde{W}$ has the desired properties (given by (\ref{formwtilde}) and 
(\ref{estimatewtilde})). This completes the proof.
\end{proof}

Next we introduce an intermediate result. 

\begin{definition}{\rm  Let $(a_n)$ and $(b_n)$ be strictly increasing sequence of positive integers. A function 
$f\in  C_0^{\infty}(\mathbb{R})$ belongs to the class $\mathcal{U}(C_0^{\infty}(\mathbb{R}),(a_n),(b_n))$ if for every compact set $K\subset \mathbb{R}$ and for every continuous function $h:\mathbb{R}\rightarrow\mathbb{R}$ vanishing at 
zero, there exists a subsequence of positive integers $(\mu_n)$ such that 
$$\sup_{x\in K}\left\vert \sum_{k=0}^{a_{\mu_n}}\frac{f^{(k)}(0)}{k!}x^k-h(x)\right\vert\rightarrow 0\hbox{ and }
\sup_{x\in K}\left\vert \sum_{k=0}^{b_{\mu_n}}\frac{f^{(k)}(0)}{k!}x^k-h(x)\right\vert\rightarrow 0\hbox{ as }n\rightarrow +\infty.$$}
\end{definition}

\begin{proposition}\label{prop_interm} Let $(a_n)$ and $(b_n)$ be strictly increasing sequence of positive integers. Then 
the set $\mathcal{U}(C_0^{\infty}(\mathbb{R}),(a_n),(b_n))$ is $G_\delta$ and dense in $C_0^{\infty}(\mathbb{R}).$
\end{proposition}

\begin{proof} It suffices to combine the ideas of the proof of \cite[Proposition 3.2]{CT} with the 
arguments of \cite{MouNes}. Let $(f_j)$ be an enumeration of all the polynomials with coefficients in $\mathbb{Q}$ and 
$f_j(0)=0.$ Let us define the set 
$$\begin{array}{rr}E(m,j,s,n)=&\{f\in C_0^{\infty}(\mathbb{R}): \sup_{x\in [-m,m]}\left\vert
\sum_{k=1}^{a_n}\frac{f^{(k)}(0)}{k!}x^k-f_j(x)\right\vert <\frac{1}{s}\\&\hbox{ and }
\sup_{x\in [-m,m]}\left\vert\sum_{k=1}^{b_n}\frac{f^{(k)}(0)}{k!}x^k-f_j(x)\right\vert <\frac{1}{s}\}
\end{array}$$
for every $m,j,s,n\in\mathbb{N}^*.$ Observe that $E(m,j,s,n)$ is an open set and the following description holds
$$\mathcal{U}(C_0^{\infty}(\mathbb{R}),(a_n),(b_n))=\bigcap_{m,j,s\in\mathbb{N}^*}\bigcup_{n\in\mathbb{N}}
E(m,j,s,n).$$
By Baire's category theorem it suffices to show that $\cup_{n\in\mathbb{N}}E(m,j,s,n)$ is dense in 
$C_0^{\infty}(\mathbb{R}).$ To do this, let $m,j,s\in\mathbb{N}^*,$ $\varepsilon>0$ and $g$ be a polynomial. 
We seek $f\in C_0^{\infty}(\mathbb{R})$ and $n\in\mathbb{N}$ such that 
$f\in E(m,j,s,n).$ Applying the proof of \cite[Lemma 2.3]{CMM}, for any $\eta>0,$ we find $c_1,c_2,\dots,c_l$ in 
$\mathbb{R}$ such that 
$$d_{\mathbb{R}^{\mathbb{N}}}((c_1,\dots,c_l,0,\dots),0)<\eta\hbox{ and }\sup_{x\in[-m,m]}\vert p(x)+g(x)-f_j(x)\vert <\eta,$$
with $p(x)=\sum_{k=1}^lc_kx^k.$ 
Since the sequences $(a_n)$ and $(b_n)$ are strictly increasing, we fix $n\in \mathbb{N}$ such that 
$\min\{a_n,b_n\}> \max\{l,\deg (g)\}.$ Moreover the linear Borel map $T_0:C_0^{\infty}(\mathbb{R})
\rightarrow \mathbb{R}^{\mathbb{N}}, f\mapsto (f^{(k)}(0)/k!)$ is open. Hence with a previous good choice 
of $\eta<\varepsilon$ we find a function $w\in C_0^{\infty}(\mathbb{R})$ such that $T_0w=p$ 
and $d_{C_0^{\infty}(\mathbb{R})}(w,0)=d_{C_0^{\infty}(\mathbb{R})}(w+g,g)<\varepsilon.$ 
So the function $f=w+g$ does the job. 
\end{proof}

\begin{proposition} \label{prop_Fek_trunc} Let $(\lambda_n)$ be a strictly increasing sequence of positive integers. Assume that 
$\limsup_n \frac{\lambda_n}{n}=+\infty.$ Then the set $\mathcal{U}(C_0^{\infty}(\mathbb{R}),(\lambda_n))$ is 
$G_\delta$ and dense in $C_0^{\infty}(\mathbb{R})$ and contains a dense vector subspace apart from $0.$
\end{proposition}

\begin{proof} Let $(f_j)$ be an enumeration of all the polynomials with coefficients in 
$\mathbb{Q}$ vanishing at zero. Let us consider the sets
$$\begin{array}{ll}E(m,j_1,j_2,s,n)=\{f\in C_0^{\infty}(\mathbb{R}):&\sup_{x\in[-m,m]}\vert \sum_{k=0}^{\lambda_n}\frac{f^{(k)}(0)}{k!}x^k-f_{j_1}(x)\vert<
\frac{1}{s}\\&\hbox{ and }\sup_{x\in[-m,m]}\vert \sum_{k=0}^{n}\frac{f^{(k)}(0)}{k!}x^k-f_{j_2}(x)\vert<
\frac{1}{s}\}\end{array}$$
for every $m,$ $j_1,$ $j_2,$ $s,$ $n\in\mathbb{N}.$ Weirstrass approximation theorem ensures that 
$$\mathcal{U}(C_0^{\infty}(\mathbb{R}),(\lambda_n))=\bigcap_{m,j_1,j_2,s\in\mathbb{N}}\bigcup_{n\in\mathbb{N}}
E(m,j_1,j_2,s,n).$$ 
Since the sets $E(m,j_1,j_2,s,n)$ are open, according to Baire's category theorem it suffices to prove that 
$\cup_{n\in\mathbb{N}} E(m,j_1,j_2,s,n)$ is dense in $C_0^{\infty}(\mathbb{R})$ for every $m,$ $j_1,$ $j_2,$ $s\in\mathbb{N}$ to obtain that the set $\mathcal{U}(C_0^{\infty}(\mathbb{R}),(\lambda_n))$ is 
$G_\delta$ and dense in $C_0^{\infty}(\mathbb{R}).$ 
To do this, we fix $m,$ $j_1,$ $j_2,$ $s\in\mathbb{N},$ $\varepsilon>0$ and $g\in C_0^{\infty}(\mathbb{R}).$ Then it suffices to find 
$n\geq 0$ and $f\in E(m,j_1,j_2,s,n)$ such that
\begin{equation}\label{inequ_eps}
d_{C_0^{\infty}(\mathbb{R})}(f,g)<\varepsilon,
\end{equation}
where $d_{C_0^{\infty}(\mathbb{R})}$ denotes the Fr\'echet distance in $C_0^{\infty}(\mathbb{R}).$ 
By Weierstrass approximation theorem we can assume that $g$ is a polynomial with $g(0)=0.$ Since 
$\limsup_{n}\frac{\lambda_n}{n}=+\infty,$ there exists a strictly increasing sequence $(\mu_n)\subset \mathbb{N}$ such that 
$ \frac{\lambda_{\mu_n}}{\mu_n}\rightarrow +\infty$ as $n\rightarrow +\infty.$ 
We apply Lemma \ref{Bernstein_lemma2} for $h:= f_{j_1}-f_{j_2},$ $l_n=1+\mu_n$ and $m_n=\lambda_{\mu_n}.$ We obtain 
a sequence of polynomial $(P_n)$ of the form $P_n(x)=\sum_{k=1+\mu_n}^{\lambda_{\mu_n}}c_{n,k}x^k$ which converges 
to $f_{j_1}-f_{j_2}$ uniformly on $[-m,m].$ There exists $N_1\in\mathbb{N}$ such that for every $n\geq N_1$ the following inequality holds
$$\sup_{x\in [-m,m]}\vert P_n(x)-(f_{j_1}(x)-f_{j_2}(x))\vert <1/2s.$$
Observe that the linear Borel map $T_0:C_0^{\infty}(\mathbb{R})
\rightarrow \mathbb{R}^{\mathbb{N}}, f\mapsto (f^{(k)}(0)/k!)$ is open. Hence the image of every $\varepsilon/2$-neighborhood 
of $0$ in $C_0^{\infty}(\mathbb{R})$ contains some $\eta$-neighborhood 
of $0$ in $\mathbb{R}^{\mathbb{N}}.$ Moreover, by construction we have
\begin{equation}\label{val_sequ}\hbox{val}(P_n)>\mu_n\hbox{ and }\mu_n\rightarrow +\infty,\hbox{ as }n\rightarrow +\infty,
\end{equation}
where $\hbox{val}(P_n)$ denotes the valuation of the polynomial $P_n.$ So the property 
(\ref{val_sequ}) implies that the inequality $d_{\mathbb{R}^{\mathbb{N}}}(P_n,0)<\eta$ holds 
for $n$ large enough. Therefore 
one can find a positive integer $N_2>N_1$ such that for $n\geq N_2$ there exists 
$u_n\in C_0^{\infty}(\mathbb{R})$ with $T_0u_n=P_n$ and $d_{C_0^{\infty}(\mathbb{R})}(u_n,0)<\varepsilon/2.$ 
On the other hand, applying Proposition \ref{prop_interm} for $a_n=\mu_n,$ $b_n=\lambda_{\mu_n},$ 
we find a function $w\in C_0^{\infty}(\mathbb{R})$ and a sufficiently large positive integer $\nu$ with 
$\mu_{\nu}>\deg (f_{j_2})$ such that 
$$d_{C_0^{\infty}(\mathbb{R})}(w,g-f_{j_2})<\varepsilon/2,\ 
\sup_{x\in [-m,m]}\vert\sum_{k=1}^{\mu_{\nu}}\frac{w^{(k)}(0)}{k!}x^k\vert<1/2s
\hbox{ and }\sup_{x\in [-m,m]}\vert\sum_{k=1}^{\lambda_{\mu_{\nu}}}\frac{w^{(k)}(0)}{k!}x^k\vert<1/2s.$$ 
Thus the function $f:=w+u_{\nu}+f_{j_2}$ belongs to $E(m,j_1,j_2,s,\mu_{\nu})$ and satisfies inequality (\ref{inequ_eps}). Indeed we have 
$$\sup_{x\in [-m,m]}\vert\sum_{k=1}^{\mu_{\nu}}\frac{f^{(k)}(0)}{k}x^k-f_{j_2}(x)\vert=
\sup_{x\in [-m,m]}\vert\sum_{k=1}^{\mu_{\nu}}\frac{w^{(k)}(0)}{k}x^k\vert <1/2s,$$
$$\begin{array}{rcl}\displaystyle\sup_{x\in [-m,m]}\vert\sum_{k=1}^{\lambda_{\mu_{\nu}}}\frac{f^{(k)}(0)}{k}x^k-f_{j_1}(x)\vert &\leq &
\displaystyle\sup_{x\in [-m,m]}\vert\sum_{k=1}^{\mu_{\nu}}\frac{w^{(k)}(0)}{k}x^k\vert+
\sup_{x\in [-m,m]}\vert P_{\nu}(x)-(f_{j_1}(x)-f_{j_2}(x))\vert\\& <&{1}/{2s}+{1}/{2s}={1}/{s},\end{array}$$
and
$$\begin{array}{rcl}d_{C_0^{\infty}(\mathbb{R})}(f,g)=d_{C_0^{\infty}(\mathbb{R})}(w+u_{\nu}+f_{j_2},g)&\leq &
d_{C_0^{\infty}(\mathbb{R})}(w,g-f_{j_2})+d_{C_0^{\infty}(\mathbb{R})}(u_{\nu},0)\\&<&{\varepsilon}/{2}
+{\varepsilon}/{2}=\varepsilon.\end{array}$$
Hence the set $\mathcal{U}(C_0^{\infty}(\mathbb{R}),(\lambda_n))$ is 
$G_\delta$ and dense in $C_0^{\infty}(\mathbb{R}).$ Finally to prove that 
the set $\mathcal{U}(C_0^{\infty}(\mathbb{R}),(\lambda_n))$ contains a dense vector subspace, except $0,$ it 
suffices to write the analogue of Lemma \ref{lemma_double} (which will be a corollary of Proposition \ref{direct_A0}) and to follow the proof of Theorem \ref{sub_vec}.  
\end{proof}

Propositions \ref{direct_A0} and \ref{prop_Fek_trunc} can be summarized as follows.

\begin{theorem}\label{general_thm} Let $(\lambda_n)$ be a strictly increasing sequence of positive integers. The following assertions are equivalent
\begin{enumerate}
\item $\mathcal{U}(C_0^{\infty}(\mathbb{R}),(\lambda_n))$ is non-empty,
\item $\limsup_n \frac{\lambda_n}{n}=+\infty.$ 
\end{enumerate}
In addition, in the case $\limsup_n \frac{\lambda_n}{n}=+\infty,$ the set $\mathcal{U}(C_0^{\infty}(\mathbb{R}),(\lambda_n))$ is 
a $G_\delta$ and dense subset of $C_0^{\infty}(\mathbb{R})$ and contains a dense vector subspace apart from $0.$
\end{theorem}

\section{Further development and remark}\label{section_remark} The notion 
of doubly universal series has connection with that of topological multiple recurrence 
in dynamical systems. We refer the reader to \cite[page 22]{CT}. Recently Costakis and Parissis proved 
that a frequently Ces\`aro hypercyclic bounded linear operator $T$ acting on an infinite dimensional 
separable Banach space over $\mathbb{C}$ is topologically multiply 
recurrent \cite{CP}. The notion of Ces\`aro hypercyclicity for an operator $T$ was introduced in \cite{FL} and that of frequent Ces\`aro hypercyclicity in \cite{CR}. Let us introduce the set of 
frequent Ces\`aro universal series. For a power series $f=\sum_{j\geq 0}a_jz^j,$ $\sigma_n(f):=
\frac{1}{n+1}\sum_{j=0}^nS_j(f)
$ denotes the sequence of Ces\`aro means of the partial 
sums of the Taylor expansion of $f$ at $0.$ We know that the set $\mathcal{U}_{Ces}(\mathbb{D})$ of functions 
$f\in H(\mathbb{D})$ such that, for every compact set $K\subset \mathbb{C}$ 
with $K^c$ connected and $K\cap \mathbb{D}=\emptyset$ and for every function $h\in A(K),$ there exists an increasing sequence $(\lambda_n)$ of positive integers such that $\sup_{z\in K}\vert 
\sigma_{\lambda_n}(f)(z)-h(z)\vert\rightarrow 0,$ as $n\rightarrow +\infty,$ is a $G_{\delta}$-subset 
of $H(\mathbb{D})$ (see \cite{melanes1} or for instance \cite{bgnp}). 

\begin{definition}\label{def_CesFUniv}
{\rm A power series $f=\sum_{j\geq 0}a_jz^j$ of radius of convergence $1$ 
is said to be a {\it frequently Ces\`aro universal series} if for every $\varepsilon >0,$ for every compact set $K\subset \mathbb{C}\setminus \mathbb{D}$ with connected complement, and any function $h\in A(K),$ 
we have  
\[
\underline{\hbox{dens}}\left\{n\in\mathbb{N};\ \sup_{z\in K}\left\vert\sigma_n(f)(z)-h(z)\right\vert <\varepsilon\right\}>0.
\] 
}
\end{definition}
\noindent The lower and upper densities of a subset $A$ of $\mathbb{N}$ are respectively defined as follows
$$\underline{\hbox{dens}}(A)=\liminf_{N\rightarrow +\infty}\frac{\#\{n\in A:n\leq N\}}{N}\hbox{ and }
\overline{\hbox{dens}}(A)=\limsup_{N\rightarrow +\infty}\frac{\#\{n\in A:n\leq N\}}{N}$$
where as usual $\#$ denotes the cardinality of the corresponding set.

According to a recent result we know that universal Taylor series cannot be frequently universal in the sense of Definition \ref{def_CesFUniv} where we replace the 
Ces\`aro operators $\sigma_n$ by the partial sums $S_n$ \cite{MouMu1}. We state a similar result for the Ces\`aro 
universal series . 

\begin{theorem}\label{NONF_Ces} The set of frequently Ces\`aro universal series is empty.
\end{theorem}

\begin{proof} Let $f$ be in $\mathcal{U}_{Ces}(\mathbb{D}).$ According to \cite[Corollary 4.4]{bayces} $f$ 
is a universal Taylor series. Let  $K\subset\mathbb{C}\setminus\mathbb{D}$ be a compact set 
with connected complement and $h$ be a non-zero polynomial. Then Theorem 3.3 of \cite{MouMu1} ensures that there exists a subsequence 
$(\lambda_n)$ of positive integers with $\overline{\hbox{dens}}(\lambda_n)=1$ such that 
$\sup_{z\in K}\vert S_{\lambda_n}(f)(z)-h(z)\vert\rightarrow 0,$ as $n\rightarrow +\infty.$ According to Section 4 of \cite{CharMou} we have $\sup_{z\in K}\vert \sigma_{\lambda_n}(f)(z)-h(z)\vert\rightarrow 0.$ We end 
as in the proof of \cite[Theorem 3.3]{MouMu1}. Indeed define the subset $A$ of $\mathbb{N}$ by 
$$A=\{n\in\mathbb{N};\ \sup_{z\in K}\vert \sigma_{n}(f)(z)\vert <d/2\},$$
where $d=\sup_{z\in K}\vert h(z)\vert.$ Thus there exists an integer $N$ large enough, such that, 
for every $n\geq N,$ $\lambda_n\notin A.$ Let us consider the sequence $\tilde{\lambda}=(\lambda_N,\lambda_{N+1},\dots).$ 
Clearly $\overline{\hbox{dens}}(\tilde{\lambda})=1.$ So the inclusion 
$A\subset\mathbb{N}\setminus\tilde{\lambda}$ implies 
$$ \underline{\hbox{dens}}(A)\leq \underline{\hbox{dens}}(\mathbb{N}\setminus\tilde{\lambda}).$$
But we have $\underline{\hbox{dens}}(\mathbb{N}\setminus\tilde{\lambda})=
1-\overline{\hbox{dens}}(\tilde{\lambda})=0.$ Thus $f$ cannot be a frequently Ces\`aro universal series.
\end{proof}

Therefore the sequence of operators given by Ces\`aro means of sequence of operators given by 
the partial sums $(S_n)$ of the Taylor development at $0$ of functions of $H(\mathbb{D})$ 
is not frequently universal even if the sequence of operators $(S_n)$ is doubly universal. Since the 
notion of doubly universality has connection with that of topological recurrence, 
we can compare this result with the main result of \cite{CP}. 

\begin{remark}{\rm \begin{enumerate}
\item Theorem \ref{NONF_Ces} remains true in the case of Fekete universal functions. To see this, it suffices to argue as in the proof of Theorem \ref{NONF_Ces} taking into account 
the results of \cite{MouMu2}.
\item The proof of Theorem \ref{NONF_Ces} shows that all the elements $f=\sum_{k\geq 0}a_kz ^k$ 
of $\mathcal{U}_{Ces}(\mathbb{D})$ are 1-upper frequently Ces\`aro universal, i.e. for every compact set 
$K\subset \mathbb{C}\setminus \mathbb{D}$ with connected complement, and any function $h\in A(K),$ 
there exists an increasing sequence $\lambda=(\lambda_n)$ of positive integers with 
$\overline{\hbox{dens}}(\lambda)=1$ such that $\sup_{z\in K}\left\vert\sigma_{\lambda_n}(f)(z)-h(z)\right\vert\rightarrow 0$ 
as $n\rightarrow +\infty.$
\end{enumerate} }
\end{remark}

Let $(\lambda_n)$ be a strictly increasing sequence of positive integers with $(\lambda_n)\ne \mathbb{N}.$ We end the paper with the following remark, which shows that one can find examples of doubly universal series with respect to the given sequence 
$(\lambda_n)$ without additional hypothesis. For instance, let us define the set 
$\mathcal{U}(\mathbb{R}^{\mathbb{N}},(\lambda_n))$ of 
sequences $(a_n)\subset\mathbb{R}^{\mathbb{N}}$ satisfying the following universal property: 
for every pair of real numbers $(l_1,l_2)$ 
there exists a subsequence $(\mu_n)$ of positive integers such that $\vert \sum_{k=0}^{\lambda_{\mu_n}}a_k-l_1\vert \rightarrow 0$ and $\vert \sum_{k=0}^{\mu_n}a_k-l_2\vert \rightarrow 0,$ as $n\rightarrow +\infty.$ Then $\mathcal{U}(\mathbb{R}^{\mathbb{N}},(\lambda_n))$ is a $G_\delta$ and dense subset of 
$\mathbb{R}^{\mathbb{N}},$ endowed with its natural topology defined in Section \ref{S_d_infinitely}. In particular we have $\mathcal{U}(\mathbb{R}^{\mathbb{N}},(\lambda_n))\ne\emptyset.$ Indeed, let us consider 
$$E(j_1,j_2,s,n)=\left\{(a_n)\in\mathbb{R}^{\mathbb{N}}:\vert \sum_{k=0}^{\lambda_n}a_k-
r_{j_1}\vert<\frac{1}{s}\hbox{ and }\vert \sum_{k=0}^{n}a_k-r_{j_2}\vert<\frac{1}{s}\right\}$$
for every $j_1,$ $j_2,$ $s,$ $n\in\mathbb{N},$ where $(r_j)$ is an enumeration of $\mathbb{Q}.$ Obviously we have the following description
$$\mathcal{U}(\mathbb{R}^{\mathbb{N}},(\lambda_n))=\bigcap_{j_1,j_2,s\in\mathbb{N}}\bigcup_{n\in\mathbb{N}}
E(j_1,j_2,s,n).$$ 
Since the sets $E(j_1,j_2,s,n)$ are open, according to Baire's category theorem it suffices to prove that 
$\cup_{n\in\mathbb{N}} E(j_1,j_2,s,n)$ is dense in $\mathbb{R}^{\mathbb{N}}$ for every $j_1,$ $j_2,$ $s\in\mathbb{N}.$ Fix $j_1,$ $j_2,$ $s\in\mathbb{N},$ $\varepsilon>0$ and $(b_n)\in \mathbb{R}^{\mathbb{N}}.$ We seek $n\geq 0$ and $(a_k)\in E(j_1,j_2,s,n)$ such that $d_{\mathbb{R}^{\mathbb{N}}}((a_n),(b_n))<\varepsilon.$ 
Let us choose $n\in\mathbb{N}$ so that $\sum_{k\geq n}2^{-k}<\varepsilon$ and $\lambda_n>n.$ It is easy to check that the sequence $(a_k)$ defined by $a_k=b_k,$ for $k=0,\dots,n-1,$ $a_{n}=r_{j_2}-\sum_{k=0}^nb_k,$ $a_k=0$ for $k=n+1,\dots,\lambda_{n}-1$ and $a_{\lambda_n}=r_{j_1}-r_{j_2}$ does the job.


\begin{thebibliography}{99}

\bibitem{AronBeau} \textsc{R. Aron, B. Beauzamy, P. Enflo},
{\it Polynomials in many variables: Real vs Complex norms}, 
J. Approx. Theory \textbf{74} (1993) 181--198.

\bibitem{bayces} \textsc{F. Bayart}, {\it Boundary behavior and Ces\`aro means of universal Taylor series,} 
Rev. Mat. Complut. \textbf{19} (2006), no. 1, 235--247.

\bibitem{bgnp} \textsc{F. Bayart, K.-G. Grosse-Erdmann, V. Nestoridis, C. Papadimitropoulos},
{\it Abstract theory of universal series and applications}, Proc.
London Math. Soc. \textbf{96} (2008) 417-463.

\bibitem{bernal}\textsc{L. Bernal-Gonz\'alez, D. Pellegrino, J.B. Seoane-Sep\'ulveda}, \textit{Linear subsets of nonlinear sets in topological 
vector spaces}, Bull. Amer. Math. Soc. \textbf{51} (2014) 71--130.

\bibitem{Bernstein} \textsc{S. Bernstein},
{\it D\'emonstration du th\'eor\`eme de Weierstrass fond\'ee sur le calcul des probabilit\'es}, 
Commun. Soc. Math. Kharkow \textbf{2} (1912-13) 1--2.

\bibitem{CharMou}\textsc{S. Charpentier, A. Mouze}, \textit{Universal Taylor series and summability}, Rev. 
Mat. Complut. \textbf{28} (2015) 153--167.

\bibitem{CMM}\textsc{S. Charpentier, Q. Menet, A. Mouze}, \textit{Closed universal subspaces of spaces of infinitely differentiable functions}, Ann. Inst. Fourier (Grenoble) \textbf{64} (2014) no1, 297--325.

\bibitem{CostaMaria}\textsc{G. Costakis, M. Marias, V. Nestoridis}, \textit{Universal Taylor series on open subsets of $\mathbb{R}^n$}, Analysis
\textbf{26} (2006) 401--409.

\bibitem{CP}\textsc{G. Costakis, I. Parissis}, \textit{Szemer\'edi's theorem, frequent hypercyclicity and 
multiple recurrence}, Math. Scand. 
\textbf{110} (2012), no. 2,  251--272.

\bibitem{CR}\textsc{G. Costakis, I. Ruzsa}, \textit{Frequently Ces\`aro hypercyclic operators 
are hypercyclic}, preprint.

\bibitem{CT}\textsc{G. Costakis, N. Tsirivas}, \textit{Doubly universal Taylor series}, J. Approx. Theory 
\textbf{180} (2014) 21--31.

\bibitem{ge} \textsc{K-G. Grosse Erdmann}, {\it Universal families and
hypercyclic operators}, Bull. Amer. Math. Soc. (N.S.) \textbf{36}
(1999) no. 3 345--381.


\bibitem{Klimek} \textsc{M. Klimek},
{\it Pluripotential theory}, 
London Mathematical Society Monographs. New Series, 6. Oxford Science Publications. The Clarendon Press, Oxford University Press, New York, (1991).

\bibitem{FL}\textsc{F. Le\'on-Saavedra}, \textit{Operators with hypercyclic Ces\`aro means}, Studia Math. 
\textbf{152} (2002) 201--215.

\bibitem{melanes1}\textsc{A. Melas, V. Nestoridis}, \textit{Universality of Taylor Series as a generic property of holomorphic functions}, Adv. in Math., \textbf{157} (2001) 138--176.

\bibitem{MouMu1}\textsc{A. Mouze, V. Munnier}, \textit{On the frequent universality of universal Taylor series in the complex plane}, Glasg. Math. J. \textbf{59} (2017) no. 1, 109--117.

\bibitem{MouMu2}\textsc{A. Mouze, V. Munnier}, \textit{Polynomial inequalities and universal Taylor series}, 
Math. Z. \textbf{284} (2016) nà. 3-4, 919--946.

\bibitem{MouNes}\textsc{A. Mouze, V. Nestoridis}, \textit{Universality and ultradifferentiable 
functions: Fekete's Theorem}, Proc. Amer. Math. Soc. 
\textbf{138} (2010) 3945--3955.

\bibitem{Nes} \textsc{V. Nestoridis}, {\it Universal Taylor series},
Ann. Inst. Fourier (Grenoble) \textbf{46} (1996) no. 5 1293--1306.

\bibitem{Pal} \textsc{G. P\'al}, {\it Zwei kleine Bemerkungen}, Tokohu Math. J. \textbf{6}
(1914/15) 42--43.


\bibitem{Tur} \textsc{P. Tur\'an}, {\it Eine neue Methode in der Analysis und 
deren Anwendungen},
Akad\'emiai Kiad\'o, Budapest, 1953.


\end{thebibliography}
\end{document}